\documentclass[12pt,a4paper]{amsart}
\usepackage[english]{babel}
\usepackage[left=2cm,right=2cm,top=2cm,bottom=2cm]{geometry}
\usepackage[numbers,sort&compress]{natbib} 
\usepackage{enumitem}
\usepackage{hyperref}
\usepackage{mathtools}
\usepackage{esint}
\usepackage{cases}
\usepackage{mathrsfs}
\usepackage{microtype}

\theoremstyle{definition}
\newtheorem{defi}{Definition}[section]

\newtheorem*{notation}{Notation}
\theoremstyle{plain}
\newtheorem{teo}[defi]{Theorem}
\newtheorem{lem}[defi]{Lemma}
\newtheorem{pro}[defi]{Proposition}

\newcommand{\N}{\mathbb{N}}
\newcommand{\vp}{\varphi}
\newcommand{\ve}{\varepsilon}
\newcommand{\R}{\mathbb{R}}

\newcommand{\se}{\subseteq}
\newcommand{\ceq}{\coloneqq}

\newcommand{\leb}{\mathcal{L}}
\newcommand{\bv}{\operatorname{BV}}
\newcommand{\J}{\mathcal{J}}
\newcommand{\hau}{\mathcal{H}}
\newcommand{\nl}{\left\|}
\newcommand{\nr}{\right\|}
\newcommand{\loc}{{\operatorname{loc}}}

\newcommand{\wx}{\widetilde{X}}
\newcommand{\wdd}{\widetilde{d}}
\newcommand{\wb}{\widetilde{B}}
\newcommand{\A}{\mathscr{A}}

\author[M. Di Marco]{Marco Di Marco}
\address{ETH Z\"urich, Department of Mathematics, R\"amistrasse 101, 8092 Z\"urich, Switzerland}
\email{mdimarco@ethz.ch}

\title[Whitney ext. th. and intr. rectifiability of jump sets in CC spaces]{Whitney extension theorem and intrinsic rectifiability of jump sets in Carnot-Carathéodory spaces}



\thanks{The author warmly thanks Sebastiano Don and Davide Vittone for many precious discussions.  The author is supported by SNSF Starting Grant TMSGI2\textunderscore226018 and is a member of GNAMPA of INdAM}

\begin{document}

\begin{abstract}
We prove that the intrinsic jump set of every locally integrable function on an equiregular Carnot-Carathéodory space is intrinsically countably rectifiable. A key new ingredient is a local scalar Whitney extension theorem for intrinsic $C^1$ functions.
\end{abstract}

\maketitle

\section{Introduction}

The geometry of jump discontinuities is a central topic in the study of fine properties of functions of bounded variation (and, more generally, for locally integrable functions). In the Euclidean setting, the jump set of a function of bounded variation on $\R^n$ is countably $(n-1)$-rectifiable \cite{afp}. Rectifiability of the jump set, however, also holds without a bounded variation assumption. Del Nin \cite{delnin} gave a direct proof for locally integrable functions, as a consequence of a more general rectifiability theorem for points admitting a nonconstant blow-up.

One of the purposes of this paper is to prove an analogous result in equiregular Carnot-Carathéodory spaces (Definition \ref{def_cc}). Here the distance is generated by a family $X=(X_1,\dots,X_m)$ of smooth, linearly independent vector fields satisfying H\"ormander's condition. The intrinsic counterpart of a regular hypersurface is a $C^1_X$-hypersurface (Definition \ref{def_ipersup}): locally, it is the zero set of a $C^1_X$-function (Definition \ref{def_c1x}) (i.e., a  continuous function with continuous horizontal gradient $Xf=(X_1f,\dots,X_mf)$) such that its horizontal gradient is non-vanishing. Accordingly, countable $X$-rectifiability (Definition \ref{def_ipersup}) means coverage by countably many such hypersurfaces, up to a set negligible for the $(Q-1)$-dimensional Hausdorff measure associated with the Carnot-Carathéodory distance, where $Q$ is the homogeneous dimension.

The notion of an intrinsic jump point is defined through the two sides of a  $C^1_X$ hypersurface. More precisely, given $u\in L^1_\loc$, a point $p$ is an intrinsic jump point of $u$ if there exist two distinct values $u^+(p),u^-(p)\in\R$ and a function $f$ of class $C^1_X$ near $p$, with $f(p)=0$ and $X f(p)\ne0$, such that
\[
\lim_{r\to 0}
\frac{1}{\leb^n(B(p,r)\cap\{\pm f>0\})}
\int_{B(p,r)\cap\{\pm f>0\}}
|u(q)-u^\pm(p)|\,d\leb^n(q)=0.
\]
Here $B(p,r)$ denotes a Carnot-Carathéodory ball\footnote{As customary in the literature, we always assume that metric balls are bounded with respect to the Euclidean topology.}, and the condition is required for both signs. Thus, $u$ approaches distinct values in the averaged $L^1$ sense on the two sides of $\{f=0\}$. The associated horizontal normal is $\nu_u(p)=Xf(p)/|Xf(p)|$, and the condition is independent of the choice of $f$ with this oriented horizontal normal at $p$. We denote by $\J_u$ the set of all intrinsic jump points of $u$, called intrinsic jump set (Definition~\ref{def_approxjump}).

In analogy with the Euclidean case, intrinsic rectifiability plays a fundamental role in the theory of functions of intrinsic bounded variation ($\bv_X$ functions). Over the last 30 years, $\bv_X$ functions have been extensively studied: a non-exhaustive list includes \cite{sbvx,dm26,dv,dv19,garofalonhieu,gn96,CapGarAhlfors,DanGarNhi,DonMagnani,Selby,agm15,am03,as10,bmp12,cm20,dmv19,magnani02,marchi14,vittone2012,sy03,bu95,dgn98,fgw94,fssc01,fssc03}. However, the question of the intrinsic rectifiability of the intrinsic jump set remained open even for general Carnot groups. In equiregular Carnot-Carathéodory spaces, Don and Vittone \cite{dv}, under the additional assumption of property~$\mathcal R$ (which requires the countable intrinsic rectifiability of the essential boundaries of sets with locally finite intrinsic perimeter)  proved that the approximate discontinuity set of a $\bv_X$ function is countably $X$-rectifiable and consists of jump points up to a $(Q-1)$-dimensional Hausdorff negligible set. Our result, Theorem \ref{teo_introrect} below, establishes the intrinsic rectifiability of the jump set itself under the sole assumption of local integrability, without imposing property~$\mathcal R$.

\begin{teo}\label{teo_introrect}
Let $\Omega$ be an open set of an equiregular Carnot-Carathéodory space $(\R^n,X)$ of homogeneous dimension $Q$ and $u \in L^1_\loc(\Omega)$. Then the intrinsic jump set $\J_u$ is countably $X$-rectifiable, that is, there exists a countable family $(S_h)_{h\in\N}$ of $C^1_X$-hypersurfaces such that
\[
\hau^{Q-1}\left(\J_u\setminus\bigcup_{h\in\N}S_h\right)=0,
\]
where $\hau^{Q-1}$ denotes the $(Q-1)$-dimensional Hausdorff measure\footnote{With respect to the Carnot-Carathéodory distance.}.
\end{teo}

The proof is organized on a countable covering $\mathcal U$ of $(\R^n,X)$ by small open neighbourhoods, which we call \emph{good sets} (Definition~\ref{def_goodset}). Their construction builds on the uniform exponential-coordinate framework that can be found, for instance, in \cite{DonMagnani} (see also \cite{abb,b96,f14}). On each good set $U$, we fix a privileged frame $Y=(Y_1,\dots,Y_n)$ extending $X$ and work with the associated exponential coordinate maps $F_{p,Y}$, uniformly as the center $p$ varies in $U$. 

In these coordinates, we introduce the \emph{horizontal projections} (Definition \ref{def_proj})
\[
\pi_p(q)\ceq\bigl(F_{p,Y}^{-1}(q)\bigr)_H\in\R^m,
\qquad p,q\in U,
\]
where the subscript $H$ denotes the first $m$ coordinates. These maps describe the horizontal displacement of $q$ from $p$ and play the role of the Euclidean increment $q-p$ in horizontal first-order expansions. Working on a good set ensures that they are simultaneously defined for every pair of points under consideration.

In a Carnot group, $\pi_p(q)$ coincides with the horizontal component of $p^{-1}q$ in exponential coordinates. The group law therefore gives the exact identity $\pi_p(z)=\pi_p(q)+\pi_q(z)$. In the present setting, we obtain the corresponding  first-order relation
\[
\pi_p(z)-\pi_q(z)+\pi_q(p)=o(t),
\qquad
t\ceq d(p,q)+d(p,z)+d(q,z),
\]
uniformly as $t\to0$ within a fixed good set; see Lemma~\ref{lem_pi}. Together with the local bound $|\pi_p(q)|\leq C d(p,q)$, this relation allows us to compare first-order horizontal expansions at nearby base points. It also controls how the regions defined by the sign of $\langle\nu,\pi_p(\cdot)\rangle$, which describe the two sides of an approximate jump, change as the base point varies.

We now outline the rectifiability argument. Fix a good set $U$, a compact set $K\subset U\cap\Omega$, and $\eta\in\N$, and consider the jump points in $K$ satisfying $|u^+-u^-|\geq 1/\eta$. By grouping points with similar jump values and normals, and with uniform control of the approximation errors, we obtain a countable covering of this set by measurable pieces $\Gamma$ with the properties that  each $\Gamma$ is locally contained in an $X$-Lipschitz hypersurface (Definition \ref{def_liphyp}) and has finite $\hau^{Q-1}$-measure (Lemmata~\ref{lem_lip} and~\ref{lem_decomp}).

This finiteness allows us to apply Lusin's and Egorov's theorems. We select countably many compact subsets $E_h\subset\Gamma\cap\J_{u,\eta}$ covering this intersection up to a $\hau^{Q-1}$-negligible set. On each $E_h$, the jump data are continuous and the $L^1$ approximation errors tend to zero uniformly in the base point along dyadic scales. Repeating the comparison of jump profiles on these compact sets gives the stronger relation
\[
\frac{|\langle\nu_u(p),\pi_p(q)\rangle|}{d(p,q)}
\longrightarrow0
\]
uniformly for $p,q\in E_h$ as $d(p,q)\to0$ (Lemma~\ref{lem_condwhit}). The latter is precisely the condition needed to realize $\nu_u$ as the horizontal gradient of a function vanishing on $E_h$.

To carry out this last step, we prove a local scalar Whitney extension theorem, Theorem \ref{teo_introwhitney} below. Let us stress that since, to the best of our knowledge, our result is the first Whitney extension theorem for $C^1_X$ functions\footnote{Let us also mention that there are several results for the extension of horizontal curves in sub-Riemannian manifolds, see for instance \cite{ss18,js17}.} in the general setting of equiregular Carnot-Carathéodory spaces, we believe also that it is of independent interest. The direct antecedents of our scalar result are the extension theorems of Franchi, Serapioni and Serra Cassano in Heisenberg groups \cite{fssc01} and in the Carnot-group setting \cite{fssc03}. A further generalization, also in the setting of Carnot groups, was proved by Vodop'yanov and Pupyshev \cite{vp06,vp06bis}.

\begin{teo}\label{teo_introwhitney}
Let $U$ be a good set of an equiregular Carnot-Carathéodory space $(\R^n,X)$ and $E \subset U$ be a compact set. Let $f:E \to \R$ and $a:E \to \R^m$ be continuous. Assume that 
\[
\frac{|f(q)-f(p)-\langle a(p),\pi_p(q) \rangle|}{d(p,q)} \longrightarrow 0
\]
    uniformly for $p,q \in E$ as $d(p,q) \to 0$. Then there exist an open set $V$ such that $E \subset V \subset U$, and a function $F \in C^1_X(V)$ such that 
    \[
    F \equiv f \text{ on }E \qquad \text{and}\qquad XF=a \text{ on }E.
    \]
\end{teo}

Applying the above result on each $E_h$ with $f=0$ and $a=\nu_u$ is enough to prove the countable $X$-rectifiability of $\J_u$.

The paper is organized as follows. Section \ref{sec_prel} collects the notation and the preliminary results needed in the rest of the paper. In Section \ref{sec_whitney} we prove the scalar Whitney extension theorem, Theorem \ref{teo_introwhitney}, and finally, in Section \ref{sec_main}, we prove the intrinsic rectifiability of the intrinsic jump set, Theorem \ref{teo_introrect}.

\section{Notation and preliminary results}\label{sec_prel}

\begin{defi}\label{def_cc}
Let $1\leq m \leq n$ be integers and let $X=(X_1,\dots,X_m)$ be an $m$-tuple of smooth and linearly independent vector fields on $\R^n$. We say that an absolutely continuous curve $\gamma\colon[0,T] \to \R^n$ is an \emph{$X$-subunit path} joining $p$ and $q$ if $\gamma(0)=p$, $\gamma(T)=q$ and there exist $h_1,\dots, h_m \in L^\infty([0,T])$ such that $\sum_{j=1}^m h_j^2 \leq 1$ and 
\[
\gamma'(t)=\sum_{j=1}^m h_j(t)X_j(\gamma(t)), \quad \text{for a.e.\ $t\in [0,T]$.}
\]
For every $p,q \in \R^n$ we  define  
\[
d(p,q) \ceq \inf \lbrace T>0: \text{ there exists an $X$-subunit path $\gamma$ joining $p$ and $q$}\rbrace,
\]
where we agree that $\inf \emptyset \ceq +\infty$.

By the Chow–Rashevskii Theorem, if for every $p \in \R^n$ the linear span of all iterated commutators of the vector fields $X_1,\dots, X_m$ computed at $p$ has dimension $n$ (i.e. $X_1,\dots,X_m$ satisfy the \emph{H\"ormander condition}), then $d$ is a distance: the latter implies that for every couple of points of $\R^n$ there always exists an $X$-subunit path joining them. In this case we say that $(\R^n,X)$ is a \emph{Carnot-Carathéodory space} of \emph{rank} $m$ and $d$ is the associated \emph{Carnot-Carathéodory distance}.         

For every $p \in \R^n$ and for every $i \in \N$ we denote by $\mathfrak{L}^i(p)$ the linear span of all the commutators of $X_1,\dots,X_m$ up to order $i$ computed at $p$. We say that a Carnot-Carathéodory space $(\R^n,X)$ is \emph{equiregular} if there exist natural numbers $n_0,n_1,\dots,n_s$ such that 
\[
0=n_0<n_1<\cdots<n_s=n \text{ and } \dim  \mathfrak{L}^i(p)=n_i, \quad \forall p \in \R^n, \forall i\in\{1,\dots, s\}.
\]
 The natural number $s$ is called \emph{step} of the Carnot-Carathéodory space. If $(\R^n,X)$ is equiregular, then the \emph{homogeneous dimension} is $Q \ceq \sum_{i=1}^si(n_i-n_{i-1})$.
\end{defi}

\begin{notation}
 In the following, $(\R^n,X)$ denotes an equiregular Carnot-Carathéodory space associated with the family $X=(X_1,\dots,X_m)$ and with homogeneous dimension $Q$. We use $d$ to denote the Carnot-Carathéodory distance associated with $X$, $B(\cdot,\cdot)$ to denote the associated open balls, $\hau^k$ to denote the associated Hausdorff $k$-measure, $\leb^n$ to denote the usual Lebesgue measure. As customary in the literature, we always assume that metric balls are bounded with respect to the Euclidean topology.
\end{notation}

\begin{defi}\label{def_c1x}
Let $\Omega \se (\R^n,X)$ be an open set and $f\colon \Omega \to \R$. We say that $f \in C^1_X(\Omega)$ if $f$ is continuous and its \emph{horizontal gradient}  $Xf \ceq (X_1f,\dots,X_mf)$,   in the sense of distributions, is represented by a continuous function.
\end{defi}

\begin{defi}\label{def_ipersup}
We say that $S \se (\R^n,X)$ is a \emph{$C^1_X$-hypersurface} if for every $p \in S$ there exist $r>0$ and $f \in C^1_X(B(p,r))$ such that the following facts hold:
\begin{enumerate}
\item[(i)]$ S \cap B(p,r)=\lbrace q \in B(p,r):f(q)=0 \rbrace$,
\item[(ii)]$Xf\neq 0$ on $B(p,r)$.
\end{enumerate}
We define the \emph{horizontal normal} to $S$ at $p \in S$ as
\[
\nu_S(p) \ceq \frac{Xf(p)}{|Xf(p)|}.
\]
Notice that $\nu_S(p)$ is well defined up to a sign and, in particular, it does not depend on the choice of $f$, see \cite[Corollary 2.14]{dv}. We say that $S  \se (\R^n,X)$ is \emph{countably $X$-rectifiable} if there exists a family $\lbrace S_h: h \in \N \rbrace$ of $C^1_X$-hypersurfaces such that
\[
\hau^{Q-1}\left( S \setminus \bigcup_{h \in \N}S_h \right)=0.
\]
 We define the \emph{horizontal normal} of a countably $X$-rectifiable set $S$ at $p \in S$ as 
\[
\nu_S(p) \ceq \nu_{S_h}(p) \text{ if }p \in S_h \setminus \bigcup_{k<h}S_k.
\]
Notice that $\nu_S$ is well defined, up to a sign, $\hau^{Q-1}$-a.e., see \cite[Proposition 2.18]{dv}.
\end{defi}
In the following we will also use the weaker notion of $X$-Lipschitz hypersurface.
\begin{defi}\label{def_liphyp}
    We say that $S \se (\R^n,X)$ is an \emph{$X$-Lipschitz hypersurface} if for every $p \in S$ there exist $R>0$ and a Lipschitz map $f:B(p,R) \to \R$ such that 
    \begin{enumerate}
        \item[(i)]$B(p,R) \cap S=\{q \in B(p,R):f(q)=0\}$;
        \item[(ii)] there exist $C>0$ and $1 \leq j \leq m$ such that $X_jf \geq C$ $\leb^n$-a.e. on $B(p,R)$.
    \end{enumerate}
\end{defi}
Let us recall that hypersurfaces with $X$-Lipschitz or $C^1_X$ regularity have locally finite $(Q-1)$-dimensional Hausdorff measure, see \cite{vittone2012}.

\begin{defi}\label{def_approxjump}
Fix $p \in (\R^n,X)$, $R>0$ and $\nu \in \mathbb{S}^{m-1}$. Let $f \in C^1_X(B(p,R))$ be such that $f(p)=0$ and $\frac{Xf(p)}{|Xf(p)|}=\nu$. For every $r \in (0,R)$ we set
\[
B^+_\nu(p,r) \ceq B(p,r) \cap \lbrace f>0\rbrace,\qquad
B^-_\nu(p,r) \ceq B(p,r) \cap \lbrace f<0\rbrace.
\]
Now let $u \in L^1_{\loc}(\Omega)$ and $p \in \Omega$. We say that $u$ has an \emph{approximate $X$-jump} at $p$ if there exist $u^+,u^- \in \R$ with $u^+ \neq u^-$ and $\nu \in \mathbb{S}^{m-1}$ such that
\begin{equation}\label{eq_xjump}
\lim_{r \to 0}\frac{1}{\leb^n(B^+_\nu(p,r))}  \int_{B^+_\nu(p,r)} |u-u^+|d\leb^n=\lim_{r \to 0} \frac{1}{\leb^n(B^-_\nu(p,r))}\int_{B^-_\nu(p,r)} |u-u^-|d\leb^n=0.
\end{equation}
The \emph{jump set} $\J_u$ is defined as the set of points where $u$ has an approximate $X$-jump. Notice that condition \eqref{eq_xjump} does not depend on the choice of the function $f$ used to construct the sets  ${B^+_\nu(p,r)}$ and ${B^-_\nu(p,r)}$, see \cite[Proposition 2.26 and Remark 2.27]{dv}.
\end{defi}

In the following we will extensively use (a system of uniform) exponential coordinates of the first kind, that we recall. For further details we refer the reader to \cite[Section 2]{DonMagnani} and the references therein.

\begin{defi}\label{def_exp}
    Let $p \in (\R^n,X)$ be fixed; choose an open neighbourhood $W \se \R^n$ of $p$ and smooth vector fields $Y_1,\dots,Y_n$ such that 
    \begin{itemize}
        \item $Y_i=X_i$ for every $i=1,\dots,m$;
        \item  for every $k=1,\dots,s$ the vector fields $Y_{n_{k-1}+1},\dots,Y_{n_k}$ are chosen among the $k$-order commutators of $X_1,\dots,X_m$;
        \item for every $q \in W$ and every $k=1,\dots,s$ the set $\{Y_1(q),\dots,Y_{n_k}(q)\}$ is a basis of $\mathfrak{L}^k(q)$.
    \end{itemize}
 Thus $Y=(Y_1,\dots,Y_n)$ is a \emph{privileged frame} in the sense of \cite[Definition 2.8]{DonMagnani}: its local existence follows from equiregularity (see \cite[Remark 2.9]{DonMagnani}). For every $q \in W$ we may find a smooth diffeomorphism $F_{q,Y}:V_q \to F_{q,Y}(V_q) \se W$ defined as 
\[
F_{q,Y}(x) \ceq \exp(x_1Y_1+\cdots+x_nY_n)(q),
\]
for some open set $V_q \se \R^n$ containing the origin. We say that $(x_1,\dots,x_n)$ are \emph{exponential coordinates of the first kind centered at $q$}. We may also choose an open neighbourhood $V \se \R^n$ of $0$ and an open neighbourhood $U \se W$ of $p$ such that the smooth map $F_Y: U \times V \to W$ defined as
\[
F_Y(q,x) \ceq \exp(x_1Y_1+\cdots+x_nY_n)(q),
\]
is well defined on $U \times V$ and, by standard ODE arguments, $F_Y(q, \cdot): V \to F_Y(\{q\}\times V)$ is a smooth diffeomorphism for every $q \in U$. In this case, we say that $F_Y$ represents a \emph{system of uniform exponential coordinates of the first kind relative to the frame $Y=(Y_1,\dots,Y_n)$}.  For every $j=1,\dots,m$ we can also define
\[
\wx_j^p \ceq dF_{p,Y}^{-1}(X_j \circ F_{p,Y}).
\]
It is easy to check that $\wx^p=(\wx_1^p,\dots,\wx_m^p)$ satisfies the H\"ormander condition. We denote by $\wdd_p$ the CC distance in (a suitable open subset of) $\R^n$ associated with $\wx^p$ and by $\wb_p(x,r)$ the metric balls associated with $\wdd_p$. We observe in passing that, since $dF_{p,Y}(0)e_j=Y_j(p)$, we have $\wx_j^p(0)=e_j$ for every $j=1,\dots,m$. Moreover, it is easy to check that for every $p \in \R^n$ and sufficiently small $r>0$ one has    
\[
d(F_{p,Y}(x_1),F_{p,Y}(x_2))=\wdd_p (x_1,x_2) \qquad\text{for every }x_1,x_2 \in \wb_p(0,r),
\]
and, in particular, $F_{p,Y}(\wb_p(x,r))=B(F_{p,Y}(x),r)$.
\end{defi}

\begin{defi}\label{def_pb}
Using the same notation of Definition \ref{def_exp} we define the \emph{degree} of the coordinate $j$ by $Y_j(p) \in \mathfrak{L}^{w_j}(p) \setminus \mathfrak{L}^{w_{j}-1}(p)$ or, equivalently, by $n_{w_{j}-1}<j \leq n_{w_j}$. For every $r>0$, the \emph{anisotropic dilation} $\delta_r:\R^n \to\R^n$ is defined by
    \[
    \delta_r(x)=(r^{w_1}x_1,\dots,r^{w_i}x_i,\dots,r^{w_n}x_n).
    \]
   Given $\mathcal W \ceq \prod_{i=1}^nw_i$ we introduce the \emph{pseudo-norm}
    \[
    \nl x \nr_* \ceq \left (   \sum_{i=1}^n |x_i|^\frac{2 \mathcal W}{w_i}  \right)^\frac{1}{2 \mathcal W}.
    \]
    and the \emph{pseudo-balls} 
    \[
   \A(r) \ceq \{x \in \R^n: \nl x \nr_* < r\}.
    \]
    We observe in passing that the pseudo-norm $\nl \cdot \nr_*$ is smooth away from the origin.
\end{defi}

The following result is just a rewriting of \cite[Corollary 5.2]{DonMagnani}.
\begin{lem}\label{lem_unifrad}
In the same notation of Definitions \ref{def_exp} and \ref{def_pb} we can find radii $r_1,r_2,r_3>0$  such that 
\[
B(q,r_1) \subset F_Y(q, \mathscr{A}(r_2))\quad  \text{for every }q \in B(p,r_3). 
\]
\end{lem}
By using the previous lemma, the following definition is well-posed.

\begin{defi}\label{def_proj}
Using the same notation of Definitions \ref{def_exp}, \ref{def_pb} and Lemma \ref{lem_unifrad} we define the map
\[
\Phi(q,z) \ceq F_{q,Y}^{-1}(z).
\]
By Lemma \ref{lem_unifrad} the map is well-defined (and smooth) on the open set
\[
\mathcal O \ceq \{(q,z) \in B(p,r_3) \times W: d(q,z)<r_1\}.
\]
Moreover, for every $(q,z) \in \mathcal O$ we define the \emph{horizontal projection of $z$ from $q$} as 
\[
\pi_q(z) \ceq \Phi(q,z)_H,
\]
where, here and in the following, $\R^m \ni \ell_H$ denotes the first $m$ components of $\ell \in \R^n$.  
    \end{defi}

\begin{defi}\label{def_goodset}
    Let $p \in  \R^n$ and let $r_1,r_3$ be as in Lemma \ref{lem_unifrad}. We define
    \[
    U_p \ceq B\left(  p, \frac{\min\{r_1,r_3\}}{100} \right)
    \]
Clearly, for every $a,b \in \overline{U_p}$ the map $\pi_a(b)$ is well defined and the family $\mathcal U' \ceq \{U_p: p \in \R^n\}$ covers $\R^n$. By the second countability of $(\R^n,X)$ we can extract a countable family $\mathcal U \ceq \{U_n: n \in \N\}$ of $\mathcal U'$ which is still a covering of $\R^n$. In the following we will refer to every element of $\mathcal U$ as a \emph{good set}.
\end{defi}

\begin{notation}
    In the following we will refer to every good set $U \in \mathcal U$ without making explicit reference to the fixed point $p$ or the relevant privileged frame $Y$: in particular, we will drop the subscript $Y$ in the maps $F_{q,Y}$ and refer to that application only by $F_q$.
\end{notation}

By \cite{nsw} (see also \cite[Theorem 2.6]{dv}) we can obtain the following.

\begin{teo}\label{teo_distequiv}
    Let $U$ be a good set. Then there exist $0<c_1 \leq c_2<+\infty$ such that for every $p,q  \in U$ we have
    \[
    c_1 N_p(q) \leq d(p,q) \leq c_2N_p(q), \qquad N_p(q) \ceq \nl F_p^{-1}(q) \nr_*.
    \]
\end{teo}

We also observe the following easy properties regarding exponential coordinates.

\begin{pro}\label{pro_propr}
Let $U$ be a good set. 
\begin{itemize}
    \item[(i)]  There exists a constant $C_N>0$ such that for every, $1 \leq j \leq m$, $p, q \in U$, $p \neq q$ one has
\[
|X_jN_p(q)| \leq C_N.
\]
\item[(ii)] Let $R>0$. There exist constants $C_J>0,r_J>0$ such that for every $p \in U$, $z \in \A(R)$ and $0<r<r_J$ one has
\begin{equation}\label{eq_jac1}
C_J^{-1} \leq |\det JF_p(\delta_rz)| \leq C_J.
\end{equation}
Consequently, for every non-negative Borel function $h$ one has 
\begin{equation}\label{eq_jac2}
\int_{F_p(\delta_r \A(R))}h(y)dy=r^Q \int_{\A(R)}h(F_p(\delta_rz))|\det JF_p(\delta_rz)|dz.
\end{equation}
\end{itemize}
\end{pro}
\begin{proof}
For (i), write, in uniform exponential coordinates of the first kind,
\[
\widetilde X_j^p=\sum_{i=1}^n b_{ji}(x)\partial_{x_i}.
\]
By \cite[Section 2]{mpv18} (see also \cite[Theorem 2.16 and Remark 2.17]{DonMagnani}) we obtain (since the pseudo-norm defined in Definition \ref{def_pb} is equivalent to the one defined in \cite{mpv18}), together with the uniformity argument in the proof of \cite[Theorem 4.6]{DonMagnani},
\[
|b_{ji}(x)|\le C\|x\|_*^{w_i-1},\qquad j=1,\dots,m,\quad i=1,\dots,n,
\]
for some constant $C>0$, uniformly in $p$. Since $\|\cdot\|_*$ is smooth away from the origin and is $\delta$-homogeneous of degree one, $\partial_{x_i}\|\cdot\|_*$ is $\delta$-homogeneous of degree $1-w_i$. Hence
\[
|\partial_{x_i}\|x\|_*|
   \le C\|x\|_*^{1-w_i},\qquad x\ne0.
\]
It follows that
\[
|\widetilde X_j\|x\|_*|
 \le \sum_{i=1}^n
 |b_{ji}(x)|\,|\partial_{x_i}\|x\|_*|
 \le C_N .
\]
Taking $x=F_p^{-1}(q)$ proves (i). For (ii), recall that
\[
dF_{p,Y}(0)e_i=Y_i(p),\qquad i=1,\dots,n,
\]
and therefore
\[
|\det JF_{p,Y}(0)| =|\det(Y_1(p),\dots,Y_n(p))|>0.
\]
Therefore, by the compactness of $\overline U$ we obtain
\[
0<c\le |\det JF_p(0)|\le C<+\infty
\qquad\text{for every }p\in\overline U.
\]
As $\delta_r z\to0$ uniformly for $z\in\A(R)$ as $r\to0$, the smoothness of $(p,x)\mapsto\det JF_{p,Y}(x)$ yields $C_J\ge1$ and $r_J>0$ such that
\[
C_J^{-1}\le|\det JF_p(\delta_rz)|\le C_J
\]
for $p\in U$, $z\in\A(R)$ and $0<r<r_J$, proving \eqref{eq_jac1}. Finally,  \eqref{eq_jac2} follows from the usual change of variables formula and the fact that $|\det D\delta_r|=r^{w_1+\cdots+w_n}=r^Q$.
\end{proof}

In the following Lemmata we collect some basic properties of horizontal projections that will be pivotal in the sequel.

\begin{lem}\label{lem_pi}
Let $U$ be a good set. There exist a constant $C_\pi> 0$ and a modulus of continuity $\omega_\pi(t)\xrightarrow{t \to 0}0$, such that for every $p,q,r \in U$, $1 \leq j \leq m$ one has
 \begin{align}
\label{eq_pro1}|\pi_p(q)| &\leq C_\pi d(p,q),\\
\label{eq_pro2}|X_j(\pi_p(q))-e_{j}| &\leq \omega_\pi(d(p,q)),\\
\label{eq_pro3}|\pi_p(r)-\pi_q(r)+\pi_q(p)| &\leq \omega_\pi(t)t,
 \end{align}
 whenever $d(p,q)+d(p,r)+d(q,r) \leq t$.
\end{lem}

\begin{proof}
Since the first $m$ coordinates have weight one, for every $x\in\mathbb R^n$ one has
\[
|x_H|\le \sqrt m\|x\|_*.
\]
Therefore, by Theorem \ref{teo_distequiv},
\[
|\pi_p(q)|=|(F_p^{-1}(q))_H|\le \sqrt mN_p(q)\le \frac{\sqrt m}{c_1}d(p,q),
\]
which proves \eqref{eq_pro1}. By Definition \ref{def_proj}, $(a,z) \mapsto \pi_a(z) $ is smooth on $\overline U\times\overline U$. For $1 \leq j\le m$ define
\[
G_{j}(a,z) \ceq X_j\big(\pi(a,\cdot)\big)(z).
\]
Since
\[
dF_a^{-1}(a)X_j(a)=e_j,
\]
we have
\[
G_{j}(a,a)=e_j.
\]
The functions $G_{j}$ have uniformly bounded first derivatives on a fixed compact neighbourhood of the diagonal. Hence
\[
|G_{j}(p,q)-e_j|
 \le C|p-q|
 \le C'd(p,q),
\]
where in the last inequality we used the local comparison between the Euclidean and Carnot-Carath\'eodory distances (see for instance \cite[Theorem 2.2(i)]{dv}). This proves \eqref{eq_pro2}, for instance with a modulus $\omega_1(t)=C't$ for small $t$. To prove  \eqref{eq_pro3}, define
\[
H(a,b,c) \ceq \pi_a(c)-\pi_b(c)+\pi_b(a).
\]
Clearly $H(a,a,a)=0$. Moreover, from $\pi_a(a)=0$ we obtain
\[
D_1\pi_a(a)+D_2\pi_a(a)=0,
\]
where $D_1,D_2$ denote the derivatives (with respect to each of the two variables) of the function $(a,z) \mapsto \pi_a(z)$. It follows immediately that all first derivatives of $H$ vanish on the diagonal, that is $DH(a,a,a)=0$. A uniform second-order Taylor estimate on a compact neighbourhood of the diagonal therefore gives
\[
|H(p,q,r)|  \le C\big(|p-q|+|p-r|+|q-r|\big)^2.
\]
Using \cite[Theorem 2.2(i)]{dv}, whenever
\[
d(p,q)+d(p,r)+d(q,r)\le t
\]
we obtain
\[
|\pi_p(r)-\pi_q(r)+\pi_q(p)|\le Ct^2.
\]
Thus \eqref{eq_pro2} and \eqref{eq_pro3} hold, after taking for instance $\omega_\pi(t)=Ct$ for $t$ sufficiently small and extending it to a modulus of continuity on $[0,+\infty)$.
\end{proof}

\begin{lem}\label{lem_tr}
  Let $U$ be a good set and $R>0$.  There exist radii $R',R''>0$ and a modulus of continuity  $\omega_R(t)\xrightarrow{t \to 0}0$ such that for every $p \in U$, $a,b \in \A(R)$, $0<\rho<R'$, if we define 
    \[
   q \ceq F_p(\delta_\rho a), \qquad  s \ceq F_p(\delta_\rho b),
   \]
   we have
      \begin{equation}\label{eq_est1}
   \left| \frac{\pi_q(s)}{\rho}-(b_H-a_H)   \right| \leq \omega_R(\rho) \quad
      \end{equation}
   and
   \begin{equation}\label{eq_est2}
        s \in F_q(\delta_\rho \A(R'')).
   \end{equation}
\end{lem}

\begin{proof}
Fix a slightly larger open set $\widetilde U$ such that $\overline U\subset \subset \widetilde U$ and such that the uniform exponential coordinates used in the definition of $U$ are defined on $\widetilde U$.\footnote{This can be done, for instance, by considering $\widetilde U \ceq B(p,\frac{\min\{r_1,r_3\}}{50})$, where $p,r_1,r_3$ are as in Definition \ref{def_goodset}.} By the same proof as in Lemma \ref{lem_pi}, estimate \eqref{eq_pro3} holds on $\widetilde U$, possibly with a different modulus, which we still denote by $\omega_\pi$. By Theorem \ref{teo_distequiv}, there exist $C_R>0$ and $R_0'>0$ such that, for $p\in U$, $a,b\in\A(R)$ and $0<\rho<R_0'$,
\[
d(p,q)+d(p,s)\le C_R\rho.
\]
In particular,
\[
d(q,s)\le 2C_R\rho.
\]
After decreasing $R_0'$ if necessary, this also ensures that $q,s\in\widetilde U$. Applying \eqref{eq_pro3} to $(p,q,s)$ gives
\[
|\pi_p(s)-\pi_q(s)+\pi_q(p)|\le 4C_R\rho\omega_\pi(4C_R\rho),
\]
while \eqref{eq_pro3} applied to $(p,q,q)$ gives
\[
|\pi_p(q)+\pi_q(p)| \le 2C_R\rho\omega_\pi(2C_R\rho).
\]
Since
\[
\pi_p(q)=\rho a_H,\qquad \pi_p(s)=\rho b_H,
\]
we deduce
\[
\left| \pi_q(s)-\rho(b_H-a_H)\right|\le 6C_R\rho\omega_\pi(4C_R\rho).
\]
Thus \eqref{eq_est1} follows by setting $\omega_R(\rho)\ceq 6C_R\omega_\pi(4C_R\rho)$. Finally, applying once more Theorem \ref{teo_distequiv}, now in coordinates centered at $q$, we obtain
\[\|F_q^{-1}(s)\|_*\le Cd(q,s) \le 2CC_R\rho .
\]
Hence, choosing $R''>2CC_R$ and decreasing $R'\ceq R_0'$ if necessary, we have
\[
\delta_{1/\rho}F_q^{-1}(s)\in\A(R''),
\]
which is equivalent to \eqref{eq_est2} and concludes the proof.
\end{proof}

To conclude this section we recall that by using (uniform) exponential coordinates of the first kind we can give an equivalent definition of approximate $X$-jumps that will be useful in the sequel, see \cite[Proposition 2.26]{dv}.

\begin{pro}\label{pro_altern}
   Let $u \in L^1_\loc(\Omega,\R)$ and $p \in \Omega$. The function $u$ has an \emph{approximate $X$-jump} at $p$ if and only if, working in (uniform) exponential coordinates of the first kind, as $r \to 0$ the functions $u_{p,r} \ceq u \circ F_p \circ \delta_r$ converge in $L^1_\loc(\R^n,\R)$ to 
   \[
   w_{u^+,u^-,\nu}(y) \ceq \begin{cases}
       u^+ \qquad \text{ if }L_\nu(y)>0\\
       u^- \qquad \text{ if }L_\nu(y)<0,
   \end{cases}
   \]
   where $u^+,u^- \in \R, \nu \in \mathbb{S}^{m-1}$ are as in Definition \ref{def_approxjump} and, here and in the following, for the sake of brevity, we use the notation below. 
   \[
   L_\nu(y) \ceq \sum_{i=1}^m \nu_iy_i, \qquad \nu \in \mathbb{S}^{m-1},y \in \R^n.
   \]
\end{pro}

\section{Scalar Whitney extension theorem}\label{sec_whitney}
In this section we prove a scalar Whitney extension theorem: we follow the classical approach that can be found, for instance, in \cite[Section 6.5]{evansgariepy} and has been adapted to the setting of Carnot groups in \cite[Theorem 5.2]{fssc03}. Before doing so, we need the following auxiliary result.

\begin{teo}\label{teo_diff}
Let $U$ be a good set, $V \se U$ be an open set, $H \in C(V)$ and $b \in C(V,\R^m)$. Assume that, for every $q \in V$, one has 
 \begin{equation}\label{eq_conddiff}
 H(x)-H(q)-\langle b(q),\pi_q(x) \rangle=o(d(x,q)) \qquad \text{ as }x \to q.
 \end{equation}
 Then $H \in C^1_X(V)$ and $XH=b$ (in the sense of distributions). Conversely, every $H \in C^1_X(V)$ satisfies \eqref{eq_conddiff} with $XH$ in place of $b$.
\end{teo}
\begin{proof}
    Assume first that \eqref{eq_conddiff} holds. Fix $q \in V$ and $1 \leq j \leq m$. For $|t|$ sufficiently small set 
    \[
    \gamma_j(t) \ceq F_q(te_j).
    \]
    The curve $\gamma_j$ is the integral curve of $X_j$ such that $\gamma_j(0)=q$. Moreover, by the definition of $\pi_q$, 
\begin{equation}\label{eq_cd1}
    \pi_q(\gamma_j(t))=te_j.
\end{equation}
Since $\gamma_j|_{[0,t]}$ (for $t$ sufficiently small) is an $X$-subunit curve we obtain
\begin{equation}\label{eq_cd2}
    d(\gamma_j(t),q) \leq |t|.
\end{equation}
Applying \eqref{eq_conddiff} at $q$ and using \eqref{eq_cd1} and \eqref{eq_cd2}, we obtain
\[
\frac{H(\gamma_j(t))-H(q)}{t}=b_j(q)  +\frac{o(d(\gamma_j(t),q))}{t}\xrightarrow{t \to 0} b_j(q)
\]
The same argument can be applied at $\gamma_j(v)$ for $0 \leq v \leq t$. In fact, for $h$ small enough,
\[
F_{\gamma_j(s)}(he_j)=\gamma_j(s+h).
\]
Therefore
\[
\frac{d}{ds}H(\gamma_j(s))=b_j(\gamma_j(s)).
\]
Since $b_j \circ \gamma_j$ is continuous, by the fundamental theorem of calculus we obtain
\[
H(\gamma_j(t))-H(q)=\int_0^t b_j(\gamma_j(s))ds.
\]
Thus $b_j$ is the Lie derivative of $H$ with respect to $X_j$ (see for instance \cite{venturini} and the references therein for the relevant definitions). Hence, by \cite[Theorem 1.2]{venturini} we have that $X_jH=b_j$ in the sense of distributions. Since this holds for every $1 \leq j \leq m$ and $b$ is continuous we obtain that $H \in C^1_X(V)$ and $XH=b$ (in the sense of distributions). Conversely, let $H \in C^1_X(V)$ and fix $q \in V$. By \cite[Proposition 2.13]{dv} we have that 
\[
\sup_{\xi \in \wb_q(0,r)}\frac{|H(F_q(\xi))-H(q)-\langle XH(q),\xi_H \rangle|}{r}\xrightarrow{r \to 0} 0.
\]
For $x$ sufficiently close to $q$, put $\xi=F^{-1}_q(x)$ and $r=2d(q,x)$. Locally $d(q,x)=\widetilde{d}_q(0,\xi)$, so $\xi \in \wb_q(0,r)$ and the preceding estimate gives
\[
H(x)-H(q)-\langle XH(q),\pi_q(x) \rangle=o(d(x,q)) \qquad \text{as }x\to q,
\]
concluding the proof.
\end{proof}

\begin{teo}\label{teo_whitney}
Let $U$ be a good set and $E \subset U$ be a compact set. Let $f:E \to \R$ and $a:E \to \R^m$ be continuous. Assume that 
    \begin{equation}\label{eq_ipwhit}
\frac{|f(q)-f(p)-\langle a(p),\pi_p(q) \rangle|}{d(p,q)} \longrightarrow 0
    \end{equation}
    uniformly for $p,q \in E$ as $d(p,q) \to 0$. Then there exist an open set $V$ such that $E \subset V \subset U$, and a function $F \in C^1_X(V)$ such that 
    \[
    F \equiv f \text{ on }E \qquad \text{and}\qquad XF=a \text{ on }E.
    \]
\end{teo}
\begin{proof}
\emph{Step 1: Define a suitable family of balls.}\\
    Define 
    \[
    \theta \ceq \frac{c_2}{c_1} \geq 1,
    \]
    where $c_2$ and $c_1$ are as in Theorem \ref{teo_distequiv}. Choose $\rho>0$ so small that $\{ x \in U: d(x,E)<8\rho \} \subset \subset U$. Define
    \[
    O \ceq \{ x \in U: 0<d(x,E)<2\rho\} \quad \text{ and }\quad
    V \ceq \{x \in U: d(x,E)<\rho \}.
    \]
    For $x \in O$ write $\delta(x) \ceq d(x,E)$ and 
    \[
    r(x) \ceq \frac{\delta(x)}{20 \theta}.
    \]
    By the Vitali covering theorem we can find a countable set $S \subset O$ such that the balls $\{B(p,r(p))\}_{p \in S}$ are pairwise disjoint and
    \[
    O \subset \bigcup_{p \in S} B(p,5r(p)).
    \]
For $x \in O$ we define the set
\[
S_x \ceq \{ p \in S: B(x, 10\theta r(x)) \cap B(p, 10\theta r(p)) \neq \emptyset \}.
\]
We observe in passing that if $p \in S_x$, then $d(p,x) \leq 10 \theta (r(p)+r(x))$ and $\delta(x) \leq d(x,p)+\delta(p)$. Consequently
\[
20\theta r(x) \leq 10 \theta (r(p)+r(x))+20\theta r (p), 
\]
hence $r(x) \leq 3r(p)$. Changing the role of $p$ and $x$ we have
\begin{equation}\label{eq_whitney1}
\frac{1}{3} \leq \frac{r(p)}{r(x)} \leq 3 \qquad \text{ for every } p \in S_x.
\end{equation}
Moreover, 
\begin{equation}\label{eq_whitney2}
B(p,r(p)) \subset B(x, (40\theta+3)r(x)),
\end{equation}
and, by the local Ahlfors $Q$-regularity of CC spaces, there exist constants $C_1,C_2>0$ such that 
\begin{equation}\label{eq_whitney3}
C_1 r^Q \leq \leb^n(B(y,r)) \leq C_2r^Q
\end{equation}
for every ball $B(y,r) \subset U$. By \eqref{eq_whitney1}, \eqref{eq_whitney2} and \eqref{eq_whitney3} and recalling the fact that the balls $\{B(p,r(p))\}_{p \in S}$ are pairwise disjoint one gets that 
\begin{equation}\label{eq_whitney5}
\#S_x \leq \frac{C_2}{C_1}(120 \theta +9)^Q.
\end{equation}
For the sake of brevity we define $L \ceq \frac{C_2}{C_1}(120 \theta +9)^Q$. 

\emph{Step 2: Define an adapted partition of the unity and the extension function.}\\
Now let $\chi \in C^\infty ([0,+\infty))$ be a non-increasing function such that $0 \leq \chi \leq 1$, with 
\begin{align*}
    &\chi(t)=1 \qquad \text{for }t \leq \frac{1}{c_1},\\
    &\chi(t)=0 \qquad \text{for }t \geq \frac{3}{2c_1},
\end{align*}
and, for $p \in S$ and $x \in O$ we define
\[
g_p(x) \ceq   \chi \left(   \frac{N_p(x)}{5r(p)}     \right).
\]
The following properties of $g_p$ are easy to check:
\begin{equation}\label{eq_whitney4}
g_p  \in C^\infty(O), \quad 0\leq g_p\leq1,
\quad
g_p\equiv1\ \text{on }B(p,5r(p)),\qquad\operatorname{spt}(g_p)\se B(p,10\theta r(p)).
\end{equation}
By Proposition \ref{pro_propr}, for every $j=1,\dots,m$ and for every $x \in O$,
\[
|X_jg_p(x)| \leq \frac{\nl \chi' \nr_\infty C_N}{5r(p)} \eqqcolon \frac{C_3}{r(p)}.
\]
Moreover, if $g_p(x) \neq 0$, then $p \in S_x$, and combining the above inequality with \eqref{eq_whitney1} we obtain
\begin{equation}\label{eq_whitney6}
|X_jg_p(x)| \leq \frac{3C_3}{r(x)} \eqqcolon \frac{C_4}{r(x)}.
\end{equation}
Now we define $\sigma: O \to \R$ as 
\[
\sigma (x) \ceq \sum_{p \in S}g_p(x).
\]
By \eqref{eq_whitney5} and \eqref{eq_whitney4} the above sum is locally finite, $\sigma \in C^\infty(O)$ and $\sigma \geq 1$ on $O$. Moreover, because of \eqref{eq_whitney6}, 
\[
|X_j \sigma(x)| \leq \frac{LC_4}{r(x)} \eqqcolon \frac{C_5}{r(x)},
\]
for every $j=1,\dots,m$ and for every $x \in O$. Finally we define a partition of the unity subordinate to the family of balls $\{  B(p,10 \theta r(p)) \}_{p \in S}$ as follows
\[
v_p(x) \ceq \frac{g_p(x)}{\sigma(x)}.
\]
In fact, for every $x \in O$, it is clear that 
\begin{align}\label{eq_whitney9}
    &\sum_{p \in S}v_p(x) = 1,\\
    & \sum_{p \in S}X_j v_p(x) =0  \text{ for every }j=1,\dots,m,\notag\\
    &|X_jv_p(x)|\leq \frac{C_4+C_5}{r(x)} \eqqcolon \frac{C_6}{r(x)} \text{ for every }j=1,\dots,m,\notag\\
      &\sum_{p \in S}|X_jv_p(x)|\leq \frac{LC_6}{r(x)} \eqqcolon \frac{C_7}{r(x)}.\notag
    \end{align}
The extended function $F$ is defined, for $x\in V$,
\[
F(x)
\ceq
\begin{cases}
f(x),&x\in E,\\[1mm]
\displaystyle
\sum_{p\in S}v_p(x)\bigl[f(\widehat p)+\langle a(\widehat p),\pi_{\widehat p}(x)\rangle\bigr],
&x\in V\setminus E,
\end{cases}
\]
where, for every $p \in S$, $\hat p \in E$ is one of the points such that $d(p,\hat p)=\delta(p)$. The sum is locally finite, and hence $F$ is smooth on $V\setminus E$. To conclude the proof we have to prove that $F \in C^1_X(V)$ and that $XF=a$ on $E$. 

\emph{Step 3: We prove that}
\begin{equation}\label{eq_claimw1}
F(x)-F(q)-\langle a(q),\pi_q(x) \rangle=o(d(x,q)) \qquad  \text{as} \quad x \to q \in E.
\end{equation}
For the sake of brevity, we use the following notation: for $p \in E$ and $x \in U$ we define
\[
A_p(x) \ceq f(p)+\langle a(p), \pi_p(x) \rangle.
\]
Moreover, we define $A_* \ceq \sup_E |a|$ and the moduli of continuity
\begin{align*}
\omega_f(t)&\ceq \sup_{p,q \in E,0<d(p,q) \leq t} \frac{|f(q)-f(p)-\langle a(p),\pi_p(q)\rangle|}{d(p,q)},\\
\omega_a(t) &\ceq  \sup_{p,q \in E,d(p,q) \leq t} |a(p)-a(q)|,\\
\Omega(t) &\ceq \omega_f(t)+\omega_a(t)+\omega_\pi(t),
\end{align*}
where $\omega_\pi$ is defined as in Lemma \ref{lem_pi}. Clearly, $\Omega(t) \xrightarrow{t \to 0}0$. For $p,q \in E$ and $z \in U$ we have
\[
A_p(z)-A_q(z) =[f(p)-f(q)-\langle a(q),\pi_q(p)\rangle]+\langle a(p)-a(q),\pi_p(z) \rangle+\langle a(q), \pi_p(z)-\pi_q(z)+\pi_q(p)\rangle,
\]
and, consequently,
\begin{equation}\label{eq_whitney7}
|A_p(z)-A_q(z)| \leq \omega_f(t)t+C_\pi \omega_a(t)t+A_*\omega_\pi(t)t \leq (1+C_\pi+A_*)\Omega(t)t \eqqcolon C_A \Omega(t)t,
\end{equation}
where $C_\pi$ is as in Lemma \ref{lem_pi} and $t \ceq d(p,q)+d(p,z)+d(q,z)$. We also observe in passing that 
\begin{equation}\label{eq_whitney8}
    |X_jA_p(z)-a_j(p)| \leq A_* \omega_\pi(d(p,z)).
\end{equation}
Fix $q \in E$. If $x \in E$, then \eqref{eq_claimw1} is exactly the Whitney condition \eqref{eq_ipwhit}. Let $x \in V \setminus E$. If $g_p(x) \neq 0$, then $d(p,x) \leq 10\theta r(p)=\delta(p)/2$, hence
\[
\delta(p) \leq d(p,x)+\delta(x)\leq \frac{1}{2}\delta(p)+\delta(x),
\]
therefore $\delta(p) \leq 2\delta(x) \leq 2d(x,q)$ and, consequently, $d(p,x) \leq d(x,q)$. Therefore
\[
d(\hat p,x) \leq \delta(p)+d(p,x) \leq 3d(x,q), \qquad d(\hat p,q) \leq \delta(p)+d(p,x)+d(x,q) \leq 4d(x,q), 
\]
the latter implying
\[
d(\hat p,q)+d(\hat p,x) +d(q,x) \leq 4d(x,q)+3d(x,q)+d(x,q)=8d(x,q)
\]
and, by \eqref{eq_whitney7},
\[
|A_{\hat p}(x)-A_q(x)| \leq 8C_A \Omega(8d(x,q))d(x,q)=o(d(x,q)).
\]
By the above estimate and the fact that $\sum_p v_p \equiv 1$, we obtain
\begin{equation}\label{eq_whitney10}
|F(x)-A_q(x)| =o(d(x,q)),
\end{equation}
which proves \eqref{eq_claimw1} and, in particular, the continuity of $F$ at every point of $E$.

\emph{Step 4: Continuity of the candidate horizontal gradient and conclusion.}\\
 For $x \in V \setminus E$ choose $\hat x \in E$ such that $d(x,\hat x)=\delta(x)$. By \eqref{eq_whitney9},
\begin{equation}\label{eq_whitney11}
X_jF(x)-a_j(\hat x)=\sum_{p \in S}X_jv_p(x)[A_{\hat p}(x)-A_{\hat x}(x)]+\sum_{p \in S}v_p(x)[X_jA_{\hat p}(x)-a_j(\hat x)].
\end{equation}
In the same fashion as above we obtain 
\[
d(\hat p,\hat x)+d(\hat p,x)+d(\hat x,x) \leq 8 \delta(x),
\]
the latter implying that 
\[
|A_{\hat p}(x)-A_{\hat x}(x)| \leq 8C_A \Omega(8\delta(x))\delta(x).
\]
By \eqref{eq_whitney9} and the fact that $r(x)=\delta(x)/(20\theta)$ the first sum on the right hand side of \eqref{eq_whitney11} is bounded by $160\theta C_AC_7\Omega(8\delta(x))$. For the second sum \eqref{eq_whitney8} and the continuity of $a$ give
\[
|X_j A_{\hat p}(x)-a_j(\hat x)|\leq A_* \omega_\pi (3\delta(x))+\omega_a(4\delta(x)) \leq (A_*+1)\Omega(8\delta(x)),
\]
the latter implying that 
\[
|X_jF(x)-a_j(\hat x)| \leq (160\theta C_AC_7+A_*+1)\Omega(8\delta(x)).
\]
If $x \to q \in E$, then $\delta(x) \to 0$ and $d(\hat x,q) \leq \delta(x)+d(x,q) \to 0$, hence $a_j(\hat x) \to a_j(q)$ and, consequently,
\begin{equation}\label{eq_whitney13}
X_jF(x) \to a_j(q) \qquad \text{as }x \to q, \quad x \in V \setminus E.
\end{equation}
Define $b: V \to \R^m$ as follows
\begin{align*}
b(x) \ceq \begin{cases}
    XF(x) &x \in V\setminus E,\\
    a(x) &x \in E.
\end{cases}
\end{align*}
The function $b$ is continuous by \eqref{eq_whitney13}. At every $q \in E$, \eqref{eq_conddiff} holds with $b(q)=a(q)$ by \emph{Step 3}. If $q \in V \setminus E$, then $F$ is smooth in a neighbourhood of $q$, hence also $C^1_X$. The converse part of Theorem \ref{teo_diff} gives \eqref{eq_conddiff} with $b=XF$. Therefore the hypotheses of the direct part of Theorem \ref{teo_diff} hold on all of $V$. We conclude that $F \in C^1_X(V)$ and $XF=b$ (in the sense of distributions). In particular $F \equiv f$ and $XF \equiv a$ on $E$, which concludes the proof.

\end{proof}

\section{Intrinsic rectifiability of the jump set}\label{sec_main}

In this section we prove the countable $X$-rectifiability of the intrinsic jump set. Before doing so, we need some preliminary results.

In the next lemma, we show that if two sufficiently close points violate a suitable cone condition, then one can construct a set $\mathcal A$ that will be used in the contradiction arguments of Lemmata \ref{lem_decomp} and \ref{lem_condwhit}.

\begin{lem}\label{lem_Aset}
Let $U$ be a good set, $\theta>0$ and $\nu_p,\nu_q \in \mathbb{S}^{m-1}$. Then there exist constants
\[
r_0>0,\qquad R_0>0,\qquad c_0>0,
\]
depending only on $U$ and $\theta$, with the following property. Whenever $p,q\in U$,  $\rho\ceq d(p,q)\in(0,r_0)$ and
\begin{equation}\label{eq_1lem1}
\langle\nu_p,\pi_p(q)\rangle\geq\theta\rho   \qquad   [\text{resp. }\langle\nu_p,\pi_p(q)\rangle\leq-\theta\rho]
\end{equation}
holds, then there exists an open set $\mathcal A$ such that
\begin{align}
\label{eq_1lem2}
\mathcal A
&\se
F_p\bigl(\delta_\rho(\A(R_0))\bigr)
\cap
F_q\bigl(\delta_\rho(\A(R_0))\bigr),\\
\label{eq_1lem3}
\leb^n(\mathcal A)
&\geq c_0\rho^Q,                   \\
\label{eq_1lem4}
\langle\nu_p,\pi_p(y)\rangle
&>0    \qquad   [\text{resp. }\langle\nu_p,\pi_p(y)\rangle<0]
\qquad\text{for every }y\in\mathcal A,       \\
\label{eq_1lem5} 
\langle\nu_q,\pi_q(y)\rangle
&<0  \qquad   [\text{resp. }\langle\nu_q,\pi_q(y)\rangle>0]
\qquad\text{for every }y\in\mathcal A.
\end{align}
\end{lem}

\begin{proof}
Let $c_1>0$ be the constant appearing in Theorem \ref{teo_distequiv}. For $\nu\in\mathbb S^{m-1}$, we denote by $(\nu,0)\in\R^n$ the vector whose first $m$ components coincide with $\nu$ and whose remaining components vanish. Consider the compact set
\[
\mathcal C_\theta
\ceq
\left\{
 a-\frac{\theta}{2}(\nu,0)+e:
 \nl a\nr_*\leq c_1^{-1},\ 
 \nu\in\mathbb S^{m-1},\ 
 |e|\leq\frac{\theta}{32}
\right\}
\cup
\left\{a\in\R^n:\nl a\nr_*\leq c_1^{-1}\right\}.
\]
Choose $R_1>0$ so large that
\begin{equation}\label{eq_1lem_R1}
\mathcal C_\theta\se\A(R_1).
\end{equation}
Apply Lemma \ref{lem_tr} with $R=R_1$ and let $R',R''>0$ and $\omega_{R_1}$ be the corresponding constants and modulus. Choose $r_1>0$ such that
\begin{equation}\label{eq_1lem10}
\omega_{R_1}(r)\leq\frac{\theta}{8}
\qquad\text{for every }0<r<r_1.
\end{equation}
Set
\[
R_0 \ceq 1+\max\{R_1,R''\}.
\]
Apply Proposition \ref{pro_propr}(ii) with $R=R_0$, and let $C_J>0$ and $r_J>0$ be the corresponding constants. Define
\[
r_0
\ceq
\frac12\min\{R',r_1,r_J,1\}.
\]
Let $p,q\in U$ and let
\[
\rho\ceq d(p,q)\in(0,r_0).
\]
Since $U$ is a good set, the point $F_p^{-1}(q)$ is well defined. We may therefore set
\[
z_q\ceq\delta_{1/\rho}\bigl(F_p^{-1}(q)\bigr).
\]
By Theorem \ref{teo_distequiv} and the homogeneity of the pseudo-norm,
\[
\rho=d(p,q)\geq c_1\nl F_p^{-1}(q)\nr_*=c_1\nl\delta_\rho z_q\nr_*=c_1\rho\nl z_q\nr_*,
\]
the latter implying
\begin{equation}\label{eq_1lem_zq_bound}
\nl z_q\nr_*\leq c_1^{-1}.
\end{equation}
Since the first $m$ coordinates have weight one,
\[
\pi_p(q)
=\bigl(F_p^{-1}(q)\bigr)_H
=(\delta_\rho z_q)_H
=\rho(z_q)_H.
\]
Thus \eqref{eq_1lem1} gives
\begin{equation}\label{eq_1lem6}
\langle\nu_p,(z_q)_H\rangle\geq\theta.
\end{equation}
Define
\[
\bar z\ceq z_q-\frac{\theta}{2}(\nu_q,0) \quad \text{ and }\quad
B_{p,q}\ceq B_E(\bar z,\theta/32),
\]
where with $B_E(\cdot,\cdot)$ we denote the usual Euclidean open ball. If $z\in B_{p,q}$, then, writing $e\ceq z-\bar z$, we have $|e_H|\leq|e|<\theta/32$ and
\[
z_H=(z_q)_H-\frac{\theta}{2}\nu_q+e_H.
\]
Consequently, by \eqref{eq_1lem6},
\begin{align}
\label{eq_1lem7}
\langle\nu_p,z_H\rangle
&>
\langle\nu_p,(z_q)_H\rangle
-\frac{\theta}{2}-\frac{\theta}{32}
\geq\frac{15}{32}\theta
>\frac{\theta}{4},\\
\label{eq_1lem9}
    \langle \nu_q, z_H-(z_q)_H \rangle& =  -\frac{\theta}{2}+\langle \nu_q,e_H   \rangle \leq -\frac{\theta}{2}+\frac{\theta}{32} <-\frac{\theta}{4},
\end{align}
By \eqref{eq_1lem_zq_bound}, the definition of $B_{p,q}$ and
\eqref{eq_1lem_R1},
\[
z_q\in\A(R_1),
\qquad
B_{p,q}\se\A(R_1).
\]
Define
\[
\mathcal A\ceq F_p(\delta_\rho B_{p,q}).
\]
Since
\[
B_{p,q}\subset\mathscr A(R_1)\subset\mathscr A(R_0)
\]
and $\rho<r_J$, the set $\mathcal A$ is well defined. Since $F_p$ is a diffeomorphism onto its image, $\mathcal A$ is open, and clearly
\[
\mathcal A\subset F_p\bigl(\delta_\rho\mathscr A(R_0)\bigr).
\]
Let $y\in\mathcal A$. Then $y=F_p(\delta_\rho z)$ for some $z\in B_{p,q}\se\A(R_1)$, while
\[
q=F_p(\delta_\rho z_q), \qquad z_q\in\A(R_1).
\]
Since $p\in U$, $\rho<R'$ and $z_q,z\in\mathscr A(R_1)$, Lemma \ref{lem_tr}, applied with $a=z_q$ and $b=z$, gives
\[
y\in F_q\bigl(\delta_\rho\mathscr A(R'')\bigr) \subset F_q\bigl(\delta_\rho\mathscr A(R_0)\bigr).
\]
proving \eqref{eq_1lem2}. Furthermore,
\[
\pi_p(y)=(\delta_\rho z)_H=\rho z_H,
\]
and hence, by \eqref{eq_1lem7},
\[
\langle\nu_p,\pi_p(y)\rangle
=\rho\langle\nu_p,z_H\rangle>0,
\]
proving \eqref{eq_1lem4}. Again by Lemma \ref{lem_tr} and \eqref{eq_1lem10}, 
\[
\left|
\frac{\pi_q(y)}{\rho}-\bigl(z_H-(z_q)_H\bigr)
\right|
\leq\omega_{R_1}(\rho)
\leq\frac{\theta}{8}.
\]
Therefore, using \eqref{eq_1lem9},
\[
\frac1\rho\langle\nu_q,\pi_q(y)\rangle
\leq
\langle\nu_q,z_H-(z_q)_H\rangle
+
\left|
\frac{\pi_q(y)}{\rho}-\bigl(z_H-(z_q)_H\bigr)
\right|<
-\frac{15}{32}\theta+\frac{\theta}{8}
<0,
\]
proving \eqref{eq_1lem5}. Finally, by Proposition \ref{pro_propr},
\[
\leb^n(\mathcal A)=\leb^n(F_p(\delta_\rho B_{p,q}))=\rho^Q \int_{B_{p,q}} |\det JF_p(\delta_\rho z)|dz \geq C_J^{-1}\leb^n(B_{p,q})\rho^Q.
\]
Since $\leb^n(B_{p,q})=\leb^n(B_E(0,\theta/32))$, \eqref{eq_1lem3} holds with $c_0 \ceq C_J^{-1}\leb^n(B_E(0,\theta/32))$. The case with the reverse signs is analogous.
\end{proof}

The next lemma proves that if a certain set satisfies a cone condition \eqref{eq_prolip}, then it can be covered by $X$-Lipschitz hypersurfaces.

\begin{lem}\label{lem_lip}
Let $U$ be a good set, $E\subset \subset  U$, $\nu_0\in\mathbb S^{m-1}$ and $r_E>0$. If for every $p,q\in E$ such that $0<d(p,q)<r_E$, one has 
\begin{equation}\label{eq_prolip}
\left|\langle\nu_0,\pi_p(q)\rangle\right| \leq \frac{1}{16\sqrt m}d(p,q),
\end{equation}
then every point of $E$ has a neighbourhood in which $E$ is contained in an $X$-Lipschitz hypersurface. Consequently, $\hau^{Q-1}(E)<+\infty$.
\end{lem}

\begin{proof}
Choose $j_0\in\{1,\dots,m\}$ such that
$|(\nu_0)_{j_0}|\geq m^{-1/2}$. Up to replacing $\nu_0$ with
$-\nu_0$, we may assume that
\[
a_0\ceq(\nu_0)_{j_0}\geq\frac{1}{\sqrt m}.
\]
We cover $\overline{E}$ with finitely many pairs of open sets
\[
V'_\alpha\subset\subset V_\alpha \subset\subset U
\]
such that 
\begin{equation}\label{eq_lipdiam}
\operatorname{diam}_d(V_\alpha)<r_E, \qquad m\omega_\pi\bigl(\operatorname{diam}_d(V_\alpha)\bigr) \leq\frac14,
\end{equation}
where $\omega_\pi$ is the modulus of continuity supplied by Lemma \ref{lem_pi}. Fix an index $\alpha$ such that $E\cap V'_\alpha\neq\emptyset$, and set
\[
\ell_p(x)\ceq\langle\nu_0,\pi_p(x)\rangle,
\qquad p\in E\cap V_\alpha,\quad x\in V_\alpha.
\]
For every such $p$ and $x$, Lemma \ref{lem_pi} and \eqref{eq_lipdiam} give
\begin{align*}
X_{j_0}\ell_p(x)=a_0+
\sum_{i=1}^m(\nu_0)_i
\bigl(X_{j_0}(\pi_p)_i(x)-\delta_{i,j_0}\bigr)\geq a_0-
\omega_\pi(d(p,x))\sum_{i=1}^m|(\nu_0)_i|\geq a_0-\frac{1}{4\sqrt m}
\geq\frac{3a_0}{4}.
\end{align*}
For $p\in E\cap V_\alpha$, define
\[
\psi_p(x)\ceq\ell_p(x)-\frac{a_0}{8}d(p,x),
\qquad x\in V_\alpha.
\]
The family $(\psi_p)_{p\in E\cap V_\alpha}$ is equi-Lipschitz on compact subsets of $V_\alpha$. Since $d(p,\cdot)$ is $1$-Lipschitz with respect to $d$, its horizontal derivatives satisfy $|X_jd(p,\cdot)|\leq1$ almost everywhere. It follows that
\begin{equation}\label{eq_psider}
X_{j_0}\psi_p
\geq\frac{3a_0}{4}-\frac{a_0}{8}
=\frac{5a_0}{8}
\qquad\leb^n\text{-a.e.\ on }V_\alpha
\end{equation}
for every $p\in E\cap V_\alpha$. Define
\[
f_\alpha(x)\ceq
\sup_{p\in E\cap V_\alpha}\psi_p(x),
\qquad x\in V_\alpha.
\]
As the supremum of an equi-Lipschitz family, $f_\alpha$ is locally Lipschitz. Let $x\in E\cap V'_\alpha$. Since $\psi_x(x)=0$, one has $f_\alpha(x)\geq0$. On the other hand, if $p\in E\cap V_\alpha$ and $p\neq x$, then \eqref{eq_prolip}, \eqref{eq_lipdiam}, and the fact that  $a_0/8\geq1/(8\sqrt m)$ yield
\[
\psi_p(x)
\leq
\frac{1}{16\sqrt m}d(p,x)-\frac{a_0}{8}d(p,x)
<0.
\]
Consequently,
\begin{equation}\label{eq_zero_lip}
E\cap V'_\alpha
\se
\{x\in V_\alpha:f_\alpha(x)=0\}.
\end{equation}
We claim that
\begin{equation}\label{eq_boundfa}
X_{j_0}f_\alpha\geq\frac{5a_0}{8}
\qquad\leb^n\text{-a.e.\ on }V_\alpha.
\end{equation}
Let $(p_h)_{h\in\N}$ be a countable dense subset of $E\cap V_\alpha$. For every fixed $x\in V_\alpha$, the map $p\mapsto\psi_p(x)$ is continuous, and hence
\[
f_\alpha(x)=\sup_{h\in\N}\psi_{p_h}(x).
\]
For $N\in\N$, set
\[
f_{\alpha,N}\ceq\max_{1\leq h\leq N}\psi_{p_h}.
\]
It is easy to check that \eqref{eq_psider} gives
\[
X_{j_0}f_{\alpha,N}\geq\frac{5a_0}{8}
\qquad\leb^n\text{-a.e.\ on }V_\alpha
\]
for every $N\in\N$. The sequence $(f_{\alpha,N})_N$ converges pointwise increasingly to $f_\alpha$ and, on every compact subset of $V_\alpha$, uniformly. Therefore, for every non-negative $\vp\in C_c^\infty(V_\alpha)$,
\[
-\int_{V_\alpha}f_\alpha\operatorname{div}(\vp X_{j_0})\,d\leb^n=-\lim_{N\to+\infty}
\int_{V_\alpha}f_{\alpha,N}
\operatorname{div}(\vp X_{j_0})d\leb^n\geq\frac{5a_0}{8}
\int_{V_\alpha}\vp d\leb^n.
\]
Since $f_\alpha$ is locally Lipschitz, the latter proves \eqref{eq_boundfa}. Thus
\[
S_\alpha\ceq\{x\in V_\alpha:f_\alpha(x)=0\}
\]
is an $X$-Lipschitz hypersurface containing $E\cap V'_\alpha$; see \cite[Definition 2.12]{dv}. Moreover,
\[
\hau^{Q-1}(S_\alpha\cap\overline{V'_\alpha})<+\infty
\]
by \cite[Theorem 1.3 and Corollary 4.14]{vittone2012}. Since finitely many sets $V'_\alpha$ cover $\overline{E}$, the conclusion follows.
\end{proof}

In the next result we prove that we can decompose the intrinsic jump set $\J_u$ in pieces contained in sets that satisfy the cone condition \eqref{eq_prolip}. In order to do so, we will make use of the following compact notation. Let $u \in L^1_\loc(\Omega)$, $\eta \in \N$: we define
\[
\J_{u,\eta} \ceq   \{x \in \J_u :|u^+(x)-u^-(x)| \geq 1/\eta \}.
\]

\begin{lem}\label{lem_decomp}
   Let $\Omega \se (\R^n,X)$ be an open set. For every $u \in L^1_\loc(\Omega)$ and $\eta \in \N$ the set $\J_{u,\eta}$ is contained in a countable union of measurable sets $(\Gamma_{u,\eta,s})_{s \in \N}$ with the following properties: for every $s \in \N$ there exist a good set $U_s$, a compact set $K_s \subset U_s \cap \Omega$, a vector $\nu_s \in \mathbb{S}^{m-1}$ and a radius $r_s>0$ such that $\Gamma_{u,\eta,s} \se K_s$ and
   \begin{equation}\label{eq_bislip}
|\langle \nu_s,\pi_p(q) \rangle| \leq \frac{1}{16\sqrt{m}}d(p,q)
   \end{equation}
   whenever $p,q \in \Gamma_{u,\eta,s}$ and $d(p,q)<r_s$.
\end{lem}

\begin{proof}
Since the family of good sets is countable and every open subset of $\R^n$ admits a compact exhaustion, it is enough to fix a good set $U$ and a compact set $K \subset U \cap \Omega$ and prove the claim on $\J_{u,\eta} \cap K$. Apply Lemma \ref{lem_Aset} with $\theta=\frac{1}{16\sqrt{m}}$. We obtain constants
\[
r_0>0,\qquad R_0>0,\qquad c_0>0.
\]
By Proposition \ref{pro_propr}, after possibly reducing $r_0$, there exists $C_J\geq 1$ such that
\begin{equation}\label{eq_43_jac}
 C_J^{-1}\leq |\det JF_p(\delta_rz)|\leq C_J
 \qquad\text{for every }p\in K,\ z\in \A(R_0),\ 0<r<r_0.
\end{equation}
Since $K$ is compact we may assume, up to reducing $r_0$ again if necessary, that
\[
\{ F_p(\delta_r z): p \in K, z \in  \overline{\mathscr{A}(R_0)} , 0 \leq r \leq r_0  \} \subset \subset \Omega.
\]
 Let $(\nu_j)_{j\in\N}$ be a countable dense subset of $\mathbb{S}^{m-1}$ and set
\[
\ve \ceq \frac{c_0}{2^{Q+3}\eta}.
\]
We fix a Borel representative of $u$. For $j\in\N$ and $a,b\in\mathbb Q$ with $a-b\geq \frac{1}{2\eta}$, define, for $p\in K$ and $0<r<r_0$,
\begin{align}
\label{eq_43_error}
\mathcal E_{j,a,b}(p,r)
&\ceq \int_{\A(R_0)}
 \left|u(F_p(\delta_rz))-w_{a,b,\nu_j}(z)\right|
 |\det JF_p(\delta_rz)|dz\\
&=\frac{1}{r^Q}\int_{F_p(\delta_r\A(R_0))}
 \left|u(y)-w_{a,b,\nu_j}
 \big(\delta_{1/r}F_p^{-1}(y)\big)\right|d\leb^n(y).
\nonumber
\end{align}
For every fixed $r$, the map $p \to \mathcal E_{j,a,b}(p,r)$ is measurable. For every $s\in\N$ such that $2^{-s}<r_0$, set
\begin{equation}\label{eq_43_Gamma}
\Gamma_{j,a,b,s}\ceq
\bigcap_{k\geq s}
\left\{p\in K:\mathcal E_{j,a,b}(p,2^{-k})\leq\ve\right\}.
\end{equation}
The sets  $\Gamma_{j,a,b,s}$ are measurable and they cover $\J_{u,\eta}\cap K$. Indeed, let $p\in\J_{u,\eta}\cap K$. Up to replacing the jump data $(u^+(p),u^-(p),\nu_p)$ with $(u^-(p),u^+(p),-\nu_p)$, we may assume that
\[
u^+(p)>u^-(p),\qquad u^+(p)-u^-(p)\geq\frac{1}{\eta}.
\]
Set $w_p\ceq w_{u^+(p),u^-(p),\nu_p}$. The map
\[
(\alpha,\beta,\nu)\longmapsto w_{\alpha,\beta,\nu}
\]
is continuous from $\R^2\times\mathbb{S}^{m-1}$ to $L^1(\A(R_0))$. Hence we may choose $j\in\N$ and $a,b\in\mathbb Q$ such that
\begin{equation}\label{eq_43_approx_profile}
 a-b\geq\frac{1}{2\eta},
 \qquad
 C_J\int_{\A(R_0)}|w_p-w_{a,b,\nu_j}|dz<\frac\ve2.
\end{equation}
By the definition of approximate $X$-jump,
\[
\int_{\A(R_0)}|u(F_p(\delta_rz))-w_p(z)|dz\xrightarrow{r\to0}0.
\]
Combining the above convergence with \eqref{eq_43_jac} and \eqref{eq_43_approx_profile}, we obtain
\[
\mathcal E_{j,a,b}(p,r)
\leq C_J\int_{\A(R_0)}|u(F_p(\delta_rz))-w_p(z)|dz
   +C_J\int_{\A(R_0)}|w_p-w_{a,b,\nu_j}|dz
<\ve
\]
for every sufficiently small $r>0$. It follows that $p\in\Gamma_{j,a,b,s}$ for some $s\in\N$. Thus the sets in \eqref{eq_43_Gamma} cover $\J_{u,\eta}\cap K$. It remains to prove that each $\Gamma_{j,a,b,s}$ satisfies \eqref{eq_bislip}. Fix $s \in \N$, $p,q\in\Gamma_{j,a,b,s}$ and set
\[
0<\rho\ceq d(p,q)<2^{-s}.
\]
Choose $k\geq s$ such that
\begin{equation}\label{eq_43_dyadic}
 \rho\leq r\ceq2^{-k}<2\rho.
\end{equation}
Assume by contradiction that
\[
 |\langle\nu_j,\pi_p(q)\rangle|>\frac{1}{16\sqrt{m}}\rho.
\]
Suppose first that $\langle\nu_j,\pi_p(q)\rangle>\frac{1}{16\sqrt{m}}\rho$. We apply Lemma \ref{lem_Aset} with $\nu_p=\nu_q=\nu_j$. We obtain an open set $\mathcal A$ such that
\begin{align*}
\mathcal A&\se F_p(\delta_\rho \A(R_0))\cap F_q(\delta_\rho \A(R_0)),\\
 \leb^n( \mathcal A)&\geq c_0\rho^Q,\\
 \langle\nu_j,\pi_p(y)\rangle&>0,  \quad\text{for every }y\in \mathcal A.\\
 \langle\nu_j,\pi_q(y)\rangle&<0,  
 \quad\text{for every }y\in  \mathcal A.
\end{align*}
By the homogeneity of $\nl\cdot\nr_*$ and \eqref{eq_43_dyadic},
\[
\delta_\rho \A(R_0)\se\delta_r\A(R_0),
\]
and therefore
\[
\mathcal A\se F_p(\delta_r\A(R_0))\cap F_q(\delta_r\A(R_0)).
\]
Moreover, for every $y\in  \mathcal A$, 
\[
L_{\nu_j}\big(\delta_{1/r}F_p^{-1}(y)\big)=\frac1r\langle\nu_j,\pi_p(y)\rangle>0
 \text{ and }
L_{\nu_j}\big(\delta_{1/r}F_q^{-1}(y)\big)=\frac1r\langle\nu_j,\pi_q(y)\rangle<0.
\]
Consequently (see Proposition \ref{pro_altern}) for every $y\in\mathcal A$ we have
\[
w_{a,b,\nu_j}\bigl(\delta_{1/r}F_p^{-1}(y)\bigr)=a,
\qquad
w_{a,b,\nu_j}\bigl(\delta_{1/r}F_q^{-1}(y)\bigr)=b.
\]
Combining the latter with \eqref{eq_43_error} we obtain
\begin{align*}
\frac{c_0}{2\eta}\rho^Q
&\leq (a-b)\leb^n(\mathcal A)\leq \int_{\mathcal  A}|u(y)-a|d\leb^n(y)
     +\int_{ \mathcal A}|u(y)-b|d\leb^n(y)\\
&\leq r^Q\big(\mathcal E_{j,a,b}(p,r)
               +\mathcal E_{j,a,b}(q,r)\big)\leq 2\ve r^Q
 <2^{Q+1}\ve\rho^Q
 =\frac{c_0}{4\eta}\rho^Q,
\end{align*}
which is a contradiction. If $\langle\nu_j,\pi_p(q)\rangle<-\frac{1}{16\sqrt{m}}\rho$, the signs furnished by Lemma \ref{lem_Aset} are reversed and the same argument gives an analogous contradiction. Hence
\[
|\langle\nu_j,\pi_p(q)\rangle|
\leq\frac{1}{16\sqrt m}d(p,q)
\qquad\text{whenever }p,q\in\Gamma_{j,a,b,s},\quad
0<d(p,q)<2^{-s},
\]
concluding the proof (up to relabeling the sets $\Gamma_{j,a,b,s}$).
\end{proof}

We observe that Lemma \ref{lem_decomp} above also proves the intrinsic countably Lipschitz rectifiability of the intrinsic jump set. In the next result, Lemma \ref{lem_condwhit}, we will prove that one can improve the cone condition \eqref{eq_bislip}  and obtain a Whitney condition \eqref{eq_44_whitney}.

\begin{lem}\label{lem_condwhit}
  Let $\Omega \se (\R^n,X)$ be an open set and $u \in L^1_\loc(\Omega)$. Let $\Gamma$ be one of the sets $\Gamma_{u,\eta,s}$ from Lemma \ref{lem_decomp} for some $s,\eta \in \N$. Then there exists a countable family of compact sets $(E_h)_{h \in \N} \se \Gamma \cap \J_{u,\eta}$ such that 
  \begin{equation}\label{eq_hau0}
  \hau^{Q-1}\left((\Gamma \cap \J_{u,\eta}) \setminus \bigcup_{h \in \N} E_h\right)=0,
  \end{equation}
  and, moreover, the following holds. Choose an orientation on $\J_{u,\eta}$ so that 
  \[
  u^+>u^- \quad \text{ on }\quad \J_{u,\eta},
  \]
  and denote by $\nu_u$ the corresponding normal. The map
  \[
p \longmapsto (u^+(p),u^-(p),\nu_u(p))
  \]
  is continuous on every $E_h$. Moreover, for every $h \in \N$,
    \begin{equation}\label{eq_44_whitney}
    \frac{|\langle \nu_u(p), \pi_p(q) \rangle|}{d(p,q)}\longrightarrow0
    \end{equation}
    uniformly for $p,q \in E_h$, $d(p,q) \to 0$.
\end{lem}

    \begin{proof}

Let $K \subset U \cap \Omega$, where $K$ is the compact set and $U$ the good set associated with $\Gamma$ (as in the proof of Lemma \ref{lem_decomp}). The set $\Gamma$ satisfies \eqref{eq_bislip} by Lemma \ref{lem_decomp}; hence, by Lemma \ref{lem_lip}, $\hau^{Q-1}(\Gamma)<+\infty$. Moreover, by \cite[Proposition 2.28]{dv}, the map
\[
p\longmapsto\big(u^+(p),u^-(p),\nu_u(p)\big)
\]
may be chosen Borel. For every $\ell\in\N$, set
\[
\theta_\ell\ceq2^{-\ell},
\]
and apply Lemma \ref{lem_Aset} with $\theta=\theta_\ell$: denote the corresponding constants by
\[
r_\ell>0,\qquad R_\ell>0,\qquad c_\ell>0.
\] 
For all sufficiently large $k$ we  define, for $p \in \Gamma \cap \J_{u,\eta}$
\begin{equation}\label{eq_44_error}
\mathcal R_{\ell,k}(p)\ceq\int_{\A(R_\ell)}
\left|u\big(F_p(\delta_{2^{-k}}z)\big)-w_{u^+(p),u^-(p),\nu_u(p)}(z)\right|
\left|\det JF_p(\delta_{2^{-k}}z)\right|dz.
\end{equation}
By the definition of approximate $X$-jump,
\[
\mathcal R_{\ell,k}(p) \xrightarrow{k \to +\infty}0
\]
for every $p \in \Gamma \cap \J_{u,\eta}$. By using Lusin's Theorem and Egorov's Theorem, for every $h \in \N$ we can find a compact set $E_h \se \Gamma \cap \J_{u,\eta}$ such that 
\[
\hau^{Q-1}((\Gamma \cap \J_{u,\eta})\setminus E_h)<\frac{1}{h},
\]
and the map $p\longmapsto\big(u^+(p),u^-(p),\nu_u(p)\big)$ is uniformly continuous on $E_h$ and, for every $\ell \in \N$, 
\begin{equation}\label{eq_44_error_to_zero}
\mathcal R_{\ell,k}(p)\xrightarrow{k\to+\infty}0,
\end{equation}
uniformly for $p \in E_h$.  The latter clearly implies \eqref{eq_hau0}. Fix one of these compact sets and write $E\ceq E_h$. Fix $\ell \in \N$. By the uniform continuity of $u^-$ on $E$, there exists $\sigma_\ell>0$ such that, whenever $p,q \in E$ and $d(p,q)<\sigma_\ell$, one has 
\begin{equation}\label{eq_44_beta_cont}
    |u^-(p)-u^-(q)|<\frac{1}{2\eta}.
\end{equation}
By the uniform convergence in \eqref{eq_44_error_to_zero}, after decreasing $\sigma_\ell$ if necessary, we may also assume that
\begin{equation}\label{eq_44_uniform_error}
\sup_{x\in E}\mathcal R_{\ell,k}(x)<\frac{c_\ell}{2^{Q+3}\eta}
\end{equation}
for every $k$ such that $2^{-k}<2\sigma_\ell$. Without loss of generality, we may also assume $2\sigma_\ell<r_\ell$ and $F_x(\delta_r \mathscr{A}(R_\ell)) \subset \Omega$ for every $x \in E$ and $0<r<2\sigma_\ell$. Let $p,q\in E$ and set
\[
0<\rho\ceq d(p,q)<\sigma_\ell.
\]
Choose $k$ so that
\begin{equation}\label{eq_44_dyadic}
\rho\leq r\ceq2^{-k}<2\rho.
\end{equation}
Assume by contradiction that
\[
|\langle\nu_u(p),\pi_p(q)\rangle|>\theta_\ell\rho.
\]
Suppose first that
\[
\langle\nu_u(p),\pi_p(q)\rangle>\theta_\ell\rho.
\]
We apply Lemma \ref{lem_Aset}  with $\nu_p=\nu_u(p)$ and $\nu_q=\nu_u(q)$. Hence there exists an open set $\mathcal A\se\Omega$ such that
\[
\mathcal A\se F_p(\delta_\rho \A(R_\ell))\cap F_q(\delta_\rho \A(R_\ell)),
\]
\begin{equation}\label{eq_44_measure}
\leb^n(\mathcal A)\geq c_\ell\rho^Q,
\end{equation}
and
\begin{equation}\label{eq_44_signs}
\langle\nu_u(p),\pi_p(y)\rangle>0,
\qquad
\langle\nu_u(q),\pi_q(y)\rangle<0
\qquad\text{for every }y\in \mathcal  A.
\end{equation}
By the homogeneity of $\nl\cdot\nr_*$ and \eqref{eq_44_dyadic},
\[
\mathcal A\se F_p(\delta_r \A(R_\ell))\cap F_q(\delta_r \A(R_\ell)).
\]
Moreover, for every $y\in \mathcal A$, we have by \eqref{eq_44_signs}
\[
L_{\nu_u(p)}\big(\delta_{1/r}F_p^{-1}(y)\big)
=\frac1r\langle\nu_u(p),\pi_p(y)\rangle>0
\text{ and }
L_{\nu_u(q)}\big(\delta_{1/r}F_q^{-1}(y)\big)
=\frac1r\langle\nu_u(q),\pi_q(y)\rangle<0.
\]
Consequently (see Proposition \ref{pro_altern}) on $\mathcal A$ the map $w_{u^+(p),u^-(p),\nu_u(p)} \circ \delta_{1/r}F_p^{-1}$ takes the value $u^+(p)$ and the map $w_{u^+(q),u^-(q),\nu_u(q)} \circ \delta_{1/r}F_q^{-1}$ takes the value $u^-(q)$. By \eqref{eq_44_beta_cont},
\[
u^+(p)-u^-(q)=u^+(p)-u^-(p)+u^-(p)-u^-(q)\geq\frac1\eta-|u^-(p)-u^-(q)|\geq\frac{1}{2\eta}.
\]
Therefore it follows from \eqref{eq_44_measure}, the change of variables formula and
\eqref{eq_44_uniform_error} that
\begin{align*}
\frac{c_\ell}{2\eta}\rho^Q
&\leq |u^+(p)-u^-(q)|\leb^n(\mathcal A)\leq\int_{\mathcal A}|u(y)-u^+(p)|d\leb^n(y)
     +\int_{\mathcal A}|u(y)-u^-(q)|d\leb^n(y)\\
&\leq r^Q\big(\mathcal R_{\ell,k}(p)+\mathcal R_{\ell,k}(q)\big)<2^{Q+1}\rho^Q\frac{c_\ell}{2^{Q+3}\eta}
=\frac{c_\ell}{4\eta}\rho^Q,
\end{align*}
which is a contradiction. If instead
\[
\langle\nu_u(p),\pi_p(q)\rangle<-\theta_\ell\rho,
\]
Lemma \ref{lem_Aset} gives the opposite signs and the same argument gives an analogous contradiction. We have therefore proved that, for all sufficiently close $p,q\in E$,
\begin{equation}\label{eq_44_theta}
|\langle\nu_u(p),\pi_p(q)\rangle|\leq\theta_\ell d(p,q).
\end{equation}
Since $\theta_\ell=2^{-\ell}\to0$ as $\ell \to +\infty$, condition \eqref{eq_44_whitney} follows.
\end{proof}

By using the Whitney extension theorem, Theorem \ref{teo_whitney}, we can finally prove the countable $X$-rectifiability of the intrinsic jump set $\J_u$.

\begin{teo}\label{teo_rect_jump}
Let $\Omega\se(\R^n,X)$ be an open set and $u\in L^1_\loc(\Omega)$. Then the intrinsic jump set $\J_u$ is countably $X$-rectifiable.
\end{teo}

\begin{proof}
For every $\eta\in\N$, consider the decomposition of $\J_{u,\eta}$ given by Lemma \ref{lem_decomp}. Applying Lemma \ref{lem_condwhit} to every piece and relabelling the resulting compact sets with their associated good sets, we obtain a countable family of compact sets $(E_{\eta,h})_{h\in\N}$ and good sets $(U_{\eta,h})_{h\in\N}$ such that
\[
E_{\eta, h} \subset U_{\eta,h} \cap \Omega ,\qquad \hau^{Q-1}\left(\J_{u,\eta}\setminus\bigcup_{h\in\N}E_{\eta,h}\right)=0,
\]
the map $\nu_u$ is continuous on every $E_{\eta,h}$, and
\[
\frac{|\langle\nu_u(p),\pi_p(q)\rangle|}{d(p,q)}\longrightarrow0
\]
uniformly for $p,q\in E_{\eta,h}$ as $d(p,q)\to0$. Fix $\eta,h\in\N$ and apply Theorem \ref{teo_whitney} on $E_{\eta,h}$ with
\[
f\equiv0,\qquad a(p)\ceq\nu_u(p).
\]
We obtain an open set $V_{\eta,h}$ with $E_{\eta,h} \subset V_{\eta,h} \subset U_{\eta,h}$ and a function $F_{\eta,h}\in C^1_X(V_{\eta,h})$ such that
\[
F_{\eta,h}=0,\qquad XF_{\eta,h}=\nu_u
\qquad\text{on }E_{\eta,h}.
\]
Since $|XF_{\eta,h}|=1$ on $E_{\eta,h}$, the open set
\[
V'_{\eta,h}\ceq\{x\in V_{\eta,h}:|XF_{\eta,h}(x)|>1/2\}
\]
contains $E_{\eta,h}$, and
\[
S_{\eta,h}\ceq V'_{\eta,h}\cap\{F_{\eta,h}=0\}
\]
is a $C^1_X$-hypersurface containing $E_{\eta,h}$. Finally,
\[
\J_u=\bigcup_{\eta\in\N}\J_{u,\eta},
\]
and therefore
\[
\hau^{Q-1}\left(\J_u\setminus\bigcup_{\eta,h\in\N}S_{\eta,h}\right)=0,
\]
proving that $\J_u$ is countably $X$-rectifiable.
\end{proof}

\bibliographystyle{acm}
\bibliography{rect}

\end{document}